\documentclass[a4paper,11pt]{amsart}
\usepackage{amsmath,amsthm,amssymb,amsfonts,enumerate,color,enumerate}
\usepackage[pdftex]{graphicx}
\usepackage{graphicx} % To include graphics

\usepackage{verbatim}
\usepackage[spanish,USenglish]{babel} % espanol, ingles
\usepackage{color}
\usepackage{tikz}
\usepackage{graphicx}
\newtheorem{theorem}{Theorem}

\newtheorem{lemma}[theorem]{Lemma}

\newtheorem{definition}[theorem]{Definition}

\newcommand{\Rset}{\mathbb{{R}}}
\newcommand{\Zset}{\mathbb{{Z}}}
\newcommand{\Nset}{\mathbb{{N}}}

\author[Abadias]{Luciano Abadias}

\address{Centro Universitario de la Defensa, Instituto Universitario de Matem\'aticas  ́
	y Aplicaciones, 50090 Zaragoza, Spain.}
\email{labadias@unizar.es}

\author[De Le\'on-Contreras]{Marta De Le\'on-Contreras}
\address{Departamento de Matem\'aticas, Facultad de Ciencias, Universidad Aut\'onoma de Madrid, 28049 Madrid, Spain.}
\email{marta.leon@uam.es}

\author[Torrea]{Jos\'e L. Torrea}
\address{Departamento de Matem\'aticas, Facultad de Ciencias, Universidad Aut\'onoma de Madrid, 28049 Madrid, Spain.}
\email{joseluis.torrea@uam.es}
\thanks{L. Abadias is partially supported by Project MTM2016-77710-P, DGI-FEDER, of the MCYTS, Spain. Second and third authors have been partially supported by MTM2015-66157-C2-1-P, MINECO-FEDER.}

\begin{document}

% ====================================================================
% Title, authors, abstract, keywords and AMS codes
% ====================================================================

% If your title is too long to fit on the running head:
% \title[short title]{Title of your article}

% If your title is too long to fit on one line:
% \title[short title]{Title first line  \\ title second line}

\title[Schauder estimates for discrete fractional integrals]{Schauder estimates for  discrete fractional integrals}

%\author{Luciano Abadias, Marta De Le\'on-Contreras and Jos\'e L. Torrea}

% If there are two authors:
% \author{First author and second author}

% If there are three or more authors:
% \author{First author, second author,...,nth author and last author}

% If the authors' names do not fit on one line:
% \author[abbreviated list]{first line \\ second line}
% In the optional argument put something like first author et al.

\begin{abstract}
  In this note we focus on the discrete fractional integrals as a natural continuation of our previous work about nonlocal fractional derivatives, discrete and continuous. We define the discrete fractional integrals by using the semigroup theory and we
		study the regularity of {discrete fractional integrals} on the discrete H\"older spaces, which it is known in the differential equations field as the discrete Schauder estimates. 
\end{abstract}
\maketitle

% ====================================================================
% Section 1
% ====================================================================

\section{Introduction}

Fractional calculus extends the definitions of derivatives and integrals to noninteger orders. It was born in 1695, with a letter from L' H\^opital to Leibniz, where he asked what would be the derivative of order $1/2$. After this moment, a lot of authors have worked in this field, and in the last century, the interest in fractional operators and fractional differential equations has grown exponentially because of the big amount of applications it has.
	
	In this note we will focus on discrete fractional operators, in particular on the discrete fractional integrals as a continuation of our recent work \cite{ALT}. We will define the discrete fractional integrals by the semigroup theory approach and we will take some advantages of the method to get some regularity results. This point of view to treat discrete fractional operators has been recently used in \cite{BCFR},\cite{CGRTV} and \cite{Ciaurri}, among others works. { Of course, in the last years the discrete fractional integrals have also been considered in a lot of papers (see for instance \cite{A,ALMV,L} and references therein), but not from the point of view of the semigroup theory as fractional powers.}

	For $f:\Zset\rightarrow\Rset$ ,  we define
	{\sl ``the discrete derivative from the right''} and {\sl ``the discrete
		derivative from the left''} as the operators given by the formulas
	\begin{equation*}
	\delta_{\rm right}f(n)=f(n)-f(n+1)\:\: and\:\:\delta_{\rm left}f(n)=f(n)-f(n-1).
	\end{equation*}

As we did in \cite{ALT}, we shall use semigroup language as an alternative approach to discrete fractional integrals.
Given the function  $G_t(n)=e^{-t}\frac{t^{n}}{n!},\, n\in\Nset_0$,  we prove in \cite[Proposition 2.2]{ALT} that the operators
	\begin{equation*}T_{t,+}f(n)=\displaystyle\sum_{j=0}^\infty G_t(j)f(n+j), \hbox{ and } T_{t,-}f(n)=\displaystyle\sum_{j=0}^\infty G_t(j)f(n-j),\qquad t>0,\,n\in\Zset.\end{equation*}
are markovian semigroups on $\ell^p(\mathbb{Z}),    \, 1\leq p\leq \infty$,  whose infinitesimal generators are $-\delta_{\rm right}$ and $-\delta_{\rm left}$, respectively. In addition, we proved that $u(n,t) = T_{t,+}f(n)$ solves the first order Cauchy problem
\begin{equation*}
\left\{\begin{array}{ll}
	\partial_t u(n,t)+ \delta_{\rm right}u(n,t) =0,&n\in\Zset,\,t\geq 0,\\
	u(n,0)=f(n),&n\in\Zset,
\end{array} \right.
\end{equation*}
and  $v(n,t) = T_{t,-}f(n)$ satisfies the analogous Cauchy problem for $\delta_{\rm left}$.
	
We recall to the reader the following Gamma function formulas for an operator $L$, \begin{equation*}
	  L^\alpha=\frac{1}{\Gamma(-\alpha)}\int_0^\infty(e^{-t L}-1)\frac{dt}{t^{1+\alpha}} \quad \hbox{  and  }\quad
	  L^{-\alpha}=\frac{1}{\Gamma(\alpha)}\int_0^\infty e^{-t L}\frac{dt}{t^{1-\alpha}},
	  \end{equation*}
		  where $0<\alpha<1$ and $e^{-t L}$ is the associated semigroup, see \cite{Bernardis,  Stein, ST,Yosida}.
	  In particular,  we  have that the powers of order $\alpha$ of the discrete derivatives can be written by
	  \begin{align*}
	  (\delta_{\rm right})^\alpha f(n)=\frac{1}{\Gamma(-\alpha)}\int_0^\infty \frac{T_{t,+}f(n) -f(n)}{t^{1+\alpha}}dt, \  0< \alpha < 1,
	  \end{align*}
	  and
	 	  \begin{align}\label{'eq19'}
	  (\delta_{\rm right})^{-\alpha} f(n)=\frac{1}{\Gamma(\alpha)}\int_0^\infty \frac{T_{t,+}f(n)}{t^{1-\alpha}}dt, \  0< \alpha < 1,
	  \end{align}
	  whenever the integrals converges, and the corresponding formula for $(\delta_{\rm left})^\alpha, \, -1< \alpha <1.$
	
	  In order to get regularity results for the discrete fractional integrals in a more general setting, we will consider our operators on a mesh of step length $h>0$ instead of the integers mesh, that is, our functions will be defined on $\Zset_h=\{jh:\:j\in\Zset\},$ for $h>0$. Hence, for $u:\Zset_h\to\Rset,$ with $h>0$, the first order difference operators on $\Zset_{h}$ are given by
	  \begin{align*}\delta_{\rm right}u(hn)=\frac{u(hn)-u(h(n+1))}{h},\quad \delta_{\rm left}u(hn)=\frac{u(hn)-u(h(n-1))}{h},\quad n\in\Zset.\end{align*}
	 In \cite{ALT} we also prove that $\{T_{\frac{t}{h},\pm}\}_{t\geq0}$ are the associated semigroups on $\ell^{p}(\Zset_h)$.

	 The main results of this note are the discrete Schauder estimates for the discrete fractional integrals. Schauder estimates are very useful in the field of differential equations because they concern the regularity of solutions to partial differential equations. Recently, Schauder estimates have been used to get the regularity of fractional operadors in the adapted H\"older spaces, see for instance \cite{ST2}.
	
	  In our case, we need some special discrete H\"older spaces, called $C^{k,\beta}_h$. These spaces were introduced in \cite{Ciaurri}. To see the definition of these spaces, see Section \ref{regularidad}.
	   \begin{theorem}[Discrete Schauder estimates] \label{Teo} 
	   	Let $0< \beta,\alpha< 1$, and $u\in \ell_{-\alpha,h}$, see \eqref{espacio}. \begin{itemize}
	   		\item[(i)] Let $u\in C_{h}^{0,\beta}$ and $\alpha+\beta<1.$  Then $(\delta_{\rm right})^{-\alpha}u\in C_{h}^{0,\beta+\alpha}$ and \begin{align*}\lVert (\delta_{\rm right})^{-\alpha}u \rVert_{C_{h}^{0,\beta+\alpha}}\leq C\lVert u \rVert_{C_{h}^{0,\beta}}.\end{align*}
	   		\item[(ii)] Let $u\in C_{h}^{0,\beta}$ and $\alpha+\beta
	   		>1.$ Then $(\delta_{\rm right})^{-\alpha}u\in C_{h}^{1,\beta+\alpha-1}$ and \begin{align*}\lVert (\delta_{\rm right})^{-\alpha}u \rVert_{C_{h}^{1,\beta+\alpha-1}}\leq C\lVert u \rVert_{C_{h}^{0,\beta}}.\end{align*}
	   		\item[(iii)] Let $u\in C_{h}^{k,\beta}$ and assume that $k+\beta+\alpha$ is not an integer. Then $(\delta_{\rm right})^{-\alpha}u\in C_{h}^{l,s}$ where $l$ is the integer part of $k+\beta+\alpha$ and $s=k+\beta+\alpha-l$.
	   		\item[(iv)] Let $u\in\ell^\infty$. Then $(\delta_{\rm right})^{-\alpha}u\in C_{h}^{0,\alpha}$ and \begin{align*}\lVert (\delta_{\rm right})^{-\alpha}u \rVert_{C_{h}^{0,\alpha}}\leq C\lVert u \rVert_{\infty}.\end{align*}
	   		%	  		\item[(iv)] Let $u\in C_{h}^{k,\beta}$ and assume that $k+\beta-\alpha$ is not an integer, with $\alpha<k+\beta.$ Then $(\delta_{\rm right})^{\alpha}u\in C_{h}^{l,s}$ where $l$ is the integer part of $k+\beta-\alpha$ and $s=k+\beta-\alpha-l$.
	   	\end{itemize}
	   	The positive constants $C$ are independent of $h$ and $u.$
	   \end{theorem}

\section{The approach via  semigroup theory.}
Let $\alpha\in\Rset.$ Along this paper we denote \begin{align*}\Lambda^{-\alpha}(m)=\frac{\alpha(\alpha+1)\cdots(\alpha+m-1)}{m!},\quad m\in\Nset,\end{align*} and $\Lambda^{-\alpha}(0)=1.$ Note that if $\alpha\in\Rset\setminus\{0,-1,-2,\dots\}$ we have that $\Lambda^{-\alpha}(m)=\binom{m+\alpha-1}{m} $ for  $m\in\Nset_0.$ Here we highlight some properties of this kernel. Also, if $0<\alpha<1$, then $\Lambda^{-\alpha}$ is decreasing as a function of $n$, while if $-1<\alpha<0,$ we have $\sum_{n=0}^{\infty}\Lambda^{-\alpha}(n)=0,$ so $\sum_{n=1}^{\infty}\Lambda^{-\alpha}(n)=-1.$

Also, the kernel $(\Lambda^{-\alpha}(n))_{n\in\Nset_0}$ could be defined by the generating function, that is, \begin{align*}\displaystyle\sum_{n=0}^{\infty}\Lambda^{-\alpha}(n)z^n=\frac{1}{(1-z)^{\alpha}},\quad |z|<1,\end{align*} and therefore we have \begin{equation}\label{conv}\Lambda^{-(\alpha+\beta)}(n)=\sum_{j=0}^n\Lambda^{-\alpha}(n-j)\Lambda^{-\beta}(j),\quad  \alpha, \beta \in \Rset,\,n\in\Nset_0.\end{equation}

In the following, we will use the asymptotic behaviour of the sequences $\Lambda^{-\alpha}.$ It is known that for every $\alpha\in\Rset\setminus\{0,-1,-2,\dots\}$, \begin{equation}\label{asym}\Lambda^{-\alpha}(n)=\frac{1}{n^{1-\alpha}\Gamma(\alpha)}\left(1+O\left({1\over n}\right)\right),\quad n\in\Nset,\end{equation}see  {\cite[Vol.I, p.77, (1.18)]{Zygmund}}. { In the case $\alpha\in\{0,-1,-2,\ldots\}$, $\Lambda^{-\alpha}(n)=0$ for $n> -\alpha.$} To see more properties of $\{\Lambda^{-\alpha}(n)\}_{n\in \Nset_0}$ in a general setting, see \cite{Zygmund}.

As it was done in \cite{Ciaurri}, we also need to consider our functions in a particular space in order to assure the convergence of our operators. For $0<\alpha<1$, we define the space {$\ell_{-\alpha,h}$} as follows:
\begin{equation}\label{espacio}
\ell_{-\alpha, h}=\left\{u:\Zset_h\to\Rset: \:\text{for every}\;n\in \Zset,\: \sum_{m=0}^\infty \frac{|u(m\pm n)h)|}{(1+m)^{1-\alpha}}<\infty\right\}.
\end{equation}
	
Hence, by using \eqref{'eq19'}, for $0<\alpha<1$,  and $f\in \ell_{-\alpha,1},$ we have
	  \begin{align*}
	  (\delta_{\rm right})^{-\alpha} f(n)&=\frac{1}{\Gamma(\alpha)}\int_0^\infty \frac{e^{-t}\sum_{j=0}^\infty \frac{t^j}{j!}f(n+j)}{t^{1-\alpha}}dt=\sum_{j=0}^\infty f(n+j)\int_0^\infty\frac{e^{-t}t^{j+\alpha}}{j!\Gamma(\alpha)}\frac{dt}{t}\\
	  &=\sum_{j=0}^\infty f(n+j) \frac{\Gamma(\alpha+j)}{\Gamma(\alpha)j!}=\sum_{j=0}^\infty\Lambda^{-\alpha}(j) f(n+j)=\sum_{m=n}^\infty \Lambda^{-\alpha}(m-n) f(m),
	  \end{align*}
where the interchange of the sum and the integral is justified because of the integral converges absolutely. By a similar way we also get \begin{align*}(\delta_{\rm left})^{-\alpha}f(n)=\sum_{j=0}^\infty\Lambda^{-\alpha}(j) f(n-j)=\sum_{m=-\infty}^{n}\Lambda^{-\alpha}(n-m) f(m).\end{align*} Observe that, as we did in \cite{ALT}, by  proceeding similarly we get \begin{align*}(\delta_{\rm right})^\alpha f(n)=\sum_{m=n}^\infty\Lambda^\alpha(m-n) f(m),\quad (\delta_{\rm left})^\alpha f(n)=\sum_{m=-\infty}^n\Lambda^\alpha(n-m) f(m),\quad n\in \Nset_0.\end{align*}
	
Now we will consider our functions on $\Zset_h=h\Zset,$ for $h>0$. Let $u:\Zset_h\to\Rset.$ { Then, for $0<\alpha<1$,} we can write
	  \begin{equation}\label{fract1}
	  (\delta_{\rm right})^{-\alpha}u(nh)={h^{\alpha}}\sum_{m=n}^{\infty}\Lambda^{-\alpha}(m-n)u(mh), \:\; \:\:\; (\delta_{\rm left})^{-\alpha}u(nh)={h^{\alpha}}\sum^{n}_{m=-\infty}\Lambda^{-\alpha}(n-m)u(mh),
	  \end{equation}
	
	  \begin{equation}\label{fract2}
	  (\delta_{\rm right})^{\alpha}u(nh)=\frac{1}{h^{\alpha}}\sum_{m=n}^{\infty}\Lambda^\alpha(m-n) u(mh) \:\; \hbox{  and  }\:\; 	  (\delta_{\rm left})^{\alpha}u(nh)=\frac{1}{h^{\alpha}}\sum_{j=-\infty}^{n}\Lambda^\alpha(n-m)u(mh),
	  \end{equation}
whenever the series converge.

{In general, for any $\alpha>0,$ it is defined }\begin{align*}(\delta_{\rm right})^{\alpha}u=(\delta_{\rm right})^{m}(\delta_{\rm right})^{\alpha-m}u,\quad (\delta_{\rm right})^{-\alpha}u=(\delta_{\rm right})^{-m}(\delta_{\rm right})^{-(\alpha-m)}u,\end{align*} where $m=[\alpha].$ {In addition, in our case, by \eqref{conv} we have that formulas \eqref{fract1} and \eqref{fract2}  are valid for every $\alpha>0.$ } Also, by \eqref{conv} we have \begin{equation*}
(\delta_{\rm right})^{-\alpha}(\delta_{\rm right})^{\alpha}u(nh)=u(nh),\quad n\in\Zset, u\in\ell^{p}(\Zset_h).
\end{equation*}
Furthermore, for $\alpha,\beta\in\Rset$, we have \begin{equation*}
(\delta_{\rm right})^{\alpha}(\delta_{\rm right})^{\beta}u(nh)=(\delta_{\rm right})^{\alpha+\beta}u(nh),\quad n\in\Zset,
\end{equation*}
for $u$ such that the series involved in the identity converge.

\section{Regularity results of Discrete Fractional Integrals}\label{regularidad}
	
	   Following the notation in \cite{ALT} and \cite{Ciaurri}, for $l,s\in\Nset_0,$ we denote $\delta_{\rm right, \rm left}^{l,s}:=(\delta_{\rm right})^{l}(\delta_{\rm left})^{s}.$
	
	  \begin{definition}\textrm{(\cite[Definition 2.1]{Ciaurri}).} Let $0<\beta\leq 1$ and $k\in\Nset_0.$ A function $u:\Zset_{h} \to \Rset$ belongs to the discrete H\"older space $C_{h}^{k,\beta}$ if \begin{align*}[\delta_{\rm right, \rm left}^{l,s}u]_{C_{h}^{0,\beta}}=\displaystyle\sup_{m\neq j}\frac{|\delta_{\rm right, \rm left}^{l,s}u(jh)-\delta_{\rm right, \rm left}^{l,s}u(hm)|}{h^{\beta}|j-m|^{\beta}}<\infty\end{align*} for each pair $l,s\in\Nset_0$ such that $l+s=k.$ The norm in the spaces $C_{h}^{k,\beta}$ is given by \begin{align*}\lVert u\rVert_{C_{h}^{k,\beta}}=\displaystyle\max_{l+s\leq k}\sup_{m\in\Zset}|\delta_{\rm right, \rm left}^{l,s}u(mh)|+\max_{l+s=k}[\delta_{\rm right, \rm left}^{l,s}u]_{C_{h}^{0,\beta}}.\end{align*}
	  \end{definition}
	
For simplicity, we only write the following theorem for $(\delta_{\rm right})^{-\alpha}$ since it is analogous for $(\delta_{\rm left})^{-\alpha}.$

To prove this theorem we need a lemma about the kernel $\Lambda^{-\alpha}$.
	
\begin{lemma} \label{nulo}
	   \begin{enumerate}
	   	\item For every $j\in\Nset_0$, and $\alpha\in\Rset$, $\Lambda^{-\alpha}(j+1)-\Lambda^{-\alpha}(j)=\Lambda^{-(\alpha-1)}(j+1)$.
	   	%\item For $0<\alpha<1$, $\displaystyle\sum_{n=1}^\infty(\Lambda^{-\alpha}(n)-\Lambda^{-\alpha}(n-1))+ \Lambda^{-\alpha}(0)=0$.
	   	\item For every $n,l\in\Zset$, with $n>l$, and $0<\alpha<1$, \begin{align*}
	   	\sum_{m=n}^\infty(\Lambda^{-\alpha}(m-n)-\Lambda^{-\alpha}(m-l))-\sum_{m=l}^{n-1}\Lambda^{- \alpha}(m-l)=0.
	   	\end{align*}
	   \end{enumerate}
	   \end{lemma}	
\begin{proof}
At first we prove (1). Observe that $\Lambda^{-\alpha}(1)-\Lambda^{-\alpha}(0)=\alpha-1=\Lambda^{-(\alpha-1)}(1)$. Let $j\in \Nset$. We have that
	   	 	   	 \begin{align*}
	   	  \Lambda^{-\alpha}(j+1)-\Lambda^{-\alpha}(j))&=\frac{\alpha(\alpha+1)\dots(\alpha+j-1)}{j!}\left(\frac{\alpha+j}{j+1}-1 \right)=\frac{(\alpha+j-1)!}{(\alpha-1)!j!}\left( \frac{\alpha-1}{j+1}\right)\\
	   	  &=\frac{\Gamma(j+\alpha)}{(j+1)!\Gamma(\alpha-1)}=\Lambda^{-(\alpha-1)}(j+1).
	   	 \end{align*}
	   	
%We prove (2). Let $0<\alpha<1$.   Then, $-1<\alpha-1<0$ and we have $\sum_{n=0}^\infty\Lambda^{-(\alpha-1)}(n)=0$. Hence, by using the identity in (1) we get
%	   	$$\displaystyle\sum_{n=1}^\infty(\Lambda^{-\alpha}(n)-\Lambda^{-\alpha}(n-1))+ \Lambda^{-\alpha}(0)=\sum_{n=1}^\infty\Lambda^{-(\alpha-1)}(n)+\Lambda^{-(\alpha-1)}(0)=0.
%	   	$$
	   	
Now we prove (2). Let $n,l\in\Zset$, with $n>l$, and $0<\alpha<1$.	By using the identity in (1) we obtain
	   		   	\begin{align*}
	   		   	\sum_{m=n}^\infty(\Lambda^{-\alpha}(m-n)-\Lambda^{-\alpha}(m-l))&=\sum_{m=n}^\infty(\Lambda^{-\alpha}(m-n)-\Lambda^{-\alpha}(m-(n-1))+\Lambda^{-\alpha}(m-(n-1))\\&\quad\quad+\dots\
	   		   	 +\Lambda^{-\alpha}(m-l-1)-\Lambda^{-\alpha}(m-l))\\
	   		   	&=	\sum_{m=n}^\infty(-\Lambda^{-(\alpha-1)}(m-(n-1))-\Lambda^{-(\alpha-1)}(m-(n-2))\\&\quad\quad-\dots-\Lambda^{-(\alpha-1)}(m-l)).
	   		   	\end{align*}
	   	Again, as $\sum_{m=k}^\infty\Lambda^{-(\alpha-1)}(m-k)=0$, we have that
	   	\begin{align*}
	   		  & -\sum_{m=n}^\infty\Lambda^{-(\alpha-1)}(m-(n-1))=\Lambda^{-(\alpha-1)}(0)=\Lambda^{-\alpha}(0)\\
	   		 &-\sum_{m=n}^\infty\Lambda^{-(\alpha-1)}(m-(n-2))=\Lambda^{-(\alpha-1)}(0)+\Lambda^{-(\alpha-1)}(1)\\
	   		  &\vdots\\
	   		  & 	-\sum_{m=n}^\infty\Lambda^{-(\alpha-1)}(m-l)=\Lambda^{-(\alpha-1)}(0)+\Lambda^{-(\alpha-1)}(1)+\dots+\Lambda^{-(\alpha-1)}(n-l-1).\\
	   	\end{align*}
	   	Thus, using again  identity on (1) we get
	   	\begin{align*}
	   	&\sum_{m=n}^\infty(\Lambda^{-\alpha}(m-n)-\Lambda^{-\alpha}(m-l))\\
&=(n-l)\Lambda^{-\alpha}(0)+(n-l-1)\Lambda^{-(\alpha-1)}(1)+\dots +\Lambda^{-(\alpha-1)}(n-l-1)\\
&=\sum_{m=l}^{n-1}\Lambda^{- \alpha}(m-l),
	   	\end{align*}
and the result follows.
	   \end{proof}
	   \newpage

Now we can prove Theorem \ref{Teo}.

{\it Proof of Theorem \ref{Teo}}.

Let $n,l\in\Zset,$ we assume $n>l$ without loss of generality. First let $u\in C_{h}^{0,\beta}$ and $\alpha+\beta<1$. By using Lemma \ref{nulo} (2) we can write
	\begin{align*}
	h^{-\alpha}&[(\delta_{\rm right})^{-\alpha}u(nh)-(\delta_{\rm right})^{-\alpha}u(lh)]=	\sum_{m=n}^\infty\Lambda^{-\alpha}(m-n)u(mh)-\sum_{m=l}^\infty
	\Lambda^{-\alpha}(m-l) u(mh)\\
	&=\sum_{m=n}^\infty(\Lambda^{-\alpha}(m-n)-\Lambda^{-\alpha}(m-l))(u(mh)-u(lh))-\sum_{m=l}^{n-1}\Lambda^{-\alpha}(m-l)(u(mh)-u(lh))\\
	&=I+II.
		\end{align*}
On the one hand, by using estimate \eqref{asym} and the hypothesis on $u$, we get
	 \begin{align*}
	 |II|&\le C[u]_{C^{0,\beta}_h}h^{\beta}\sum_{m=l+1}^{n-1}\frac{|m-l|^\beta}{|m-l|^{1-\alpha}}= C[u]_{C^{0,\beta}_h}h^{\beta}\sum_{k=1}^{n-1-l}\frac{1}{k^{1-\alpha-\beta}}\le C[u]_{C^{0,\beta}_h}h^{\beta}(n-l)^{\alpha+\beta}.
	 \end{align*}
Before doing the estimation for $I$, observe that, as $n>l$, by \eqref{asym} we have that $|\Lambda^{-\alpha}(m-n)|\le \frac{C}{(m-n)^{1-\alpha}}$ and $|\Lambda^{-\alpha}(m-l)|\le \frac{C}{(m-n)^{1-\alpha}}$ for $m\geq n+1.$ Also, by using Lemma \ref{nulo} (1) and  \eqref{asym} we get that
	\begin{align}\label{dif}
	&|\Lambda^{-\alpha}(m-n)-\Lambda^{-\alpha}(m-l)|\nonumber\\
&=|-\Lambda^{-(\alpha-1)}(m-(n-1))-\Lambda^{-(\alpha-1)}(m-(n-2))-\dots-\Lambda^{-(\alpha-1)}(m-l)|\nonumber\\
& \le \frac{C|n-l|}{(m-(n-1))^{2-\alpha}}\le \frac{C|n-l|}{(m-n)^{2-\alpha}}\quad m\geq n+1.
	\end{align}
Hence, by using the comments above, the hypothesis on $u$ and \eqref{asym}, we obtain that
		\begin{align*}
			|I|&\le  C[u]_{C^{0,\beta}_h}h^{\beta}\left(|n-l|^\beta+\frac{|n-l|^\beta}{(n-l)^{1-\alpha}}+\sum_{m=n+1}^{2n-l}\frac{|m-l|^\beta}{(m-n)^{1-\alpha}}+ \sum_{m=2n-l+1}^\infty \frac{(n-l)|m-l|^\beta}{(m-n)^{2-\alpha}}\right)\\
		&\le C[u]_{C^{0,\beta}_h}h^{\beta}\left((|n-l|^{\alpha+\beta}+ \sum_{k=1}^{n-l}\frac{k^\beta+(n-l)^\beta}{k^{1-\alpha}}+ \sum_{k=n-l+1}^\infty \frac{(n-l)(k^\beta+(n-l)^\beta)}{k^{2-\alpha}}\right)\\
		&\le C[u]_{C^{0,\beta}_h}h^{\beta}(n-l)^{\alpha+\beta}.
	 \end{align*}

Now suppose that $u\in C_{h}^{0,\beta}$ with $\alpha+\beta>1$. By the definition of the space $C^{1,\alpha+\beta-1}_h$, we have to prove that $\delta_{\rm right}((\delta_{\rm right})^{-\alpha}u)\in C_{h}^{0,\alpha+\beta-1}$. By using $\delta_{\rm right}((\delta_{\rm right})^{-\alpha}u)=(\delta_{\rm right})^{1-\alpha}u$  and  \cite[Theorem 3.2]{ALT}, we conclude that $\delta_{\rm right}((\delta_{\rm right})^{-\alpha}u)\in C_{h}^{0,\alpha+\beta-1}$, so the result follows.

We prove statement (iii) for $k=1$. The other cases follow by iteration.

Let $u\in C^{1,\beta}_h$ and $\alpha+\beta<1$. By hypothesis, $\delta_{\rm right}u$  belongs to $C^{0,\beta}_h$. We want to prove that $\delta_{\rm right}^{-\alpha}u\in C^{1,\alpha+\beta}_h$, that is, $\delta_{\rm right}(\delta_{\rm right}^{-\alpha})u=\delta_{\rm right}^{-\alpha}(\delta_{\rm right}u)\in C^{0,\alpha+\beta}_h$, and this is consequence of (i).

 Now suppose that $u\in C^{1,\beta}_h$ and $\alpha+\beta>1$. By hypothesis, $\delta_{\rm right}u\in C^{0,\beta}_h$. We want to prove that $\delta_{\rm right}^{-\alpha}u\in C^{2,\alpha+\beta-1}_h$, that is, $(\delta_{\rm right})^2(\delta_{\rm right}^{-\alpha}u)=\delta_{\rm right}(\delta_{\rm right}^{1-\alpha})u\in C^{0,\alpha+\beta-1}_h$.
 By using (ii), we have that  $\delta_{\rm right}^{-\alpha}(\delta_{\rm right}u)=\delta_{\rm right}^{1-\alpha}u\in C^{1,\alpha+\beta-1}_h$, and by the definition of the space $C^{1,\alpha+\beta-1}_h,$ we conclude that $\delta_{\rm right}(\delta_{\rm right}^{1-\alpha})u\in C^{0,\alpha+\beta-1}_h$.

Finally,  assume that $u\in\ell^\infty$. Again, we can write
	 \begin{align*}
	 h^{-\alpha}[(\delta_{\rm right})^{-\alpha}u(nh)-(\delta_{\rm right})^{-\alpha}u(lh)]&=\sum_{m=n}^\infty(\Lambda^{-\alpha}(m-n)-\Lambda^{-\alpha}(m-l))u(mh)\\&-\sum_{m=l}^{n-1}\Lambda^{-\alpha}(m-l)u(mh).
\end{align*}

By using \eqref{dif}, we have
\begin{align*}
\left|\sum^\infty_{m=2n-l+1}(\Lambda^{-\alpha}(m-n)-\Lambda^{-\alpha}(m-l))u(mh)\right|\le \|u\|_\infty\sum_{m=2n-l+1}^\infty\frac{n-l}{(m-n)^{2-\alpha}}\le C \|u\|_\infty (n-l)^\alpha
\end{align*}
and by using \eqref{asym}, we get that
\begin{align*}
&\left|\sum_{m=n}^{2n-l}(\Lambda^{-\alpha}(m-n)-\Lambda^{-\alpha}(m-l))u(mh)\right|\\&\le  C \|u\|_\infty\left(1+\frac{1}{(n-l)^{1-\alpha}}+\sum_{m=n+1}^{2n-l}(|\Lambda^{-\alpha}(m-n)|+|\Lambda^{-\alpha}(m-l)|)\right)\\
&\le  C \|u\|_\infty\left(1+\frac{1}{(n-l)^{1-\alpha}}+\sum_{m=n+1}^{2n-l}\frac{1}{(m-n)^{1-\alpha}}\right)\le C \|u\|_\infty (n-l)^\alpha
\end{align*}
and
\begin{align*}
\left|\sum_{m=l}^{n-1} \Lambda^{-\alpha}(m-l)u(mh)\right|\le  C \|u\|_\infty\left(1+ \sum_{m=l+1}^{n-1}\frac{1}{(m-l)^{1-\alpha}}\right)\le C \|u\|_\infty (n-l)^\alpha.
\end{align*}

$\hfill {\Box}$

% ====================================================================
% Section n
% ====================================================================

% ====================================================================
% Acknowledgements
% ====================================================================

%\ack  We would like to thank the organizers of the Fourteen International Conference Zaragoza-Pau on Mathematics and Its Applications 2016 for the invitation to be part of this nice and interesting congress.

% ====================================================================
% References
% ====================================================================

% 1. Build your bib file and write its name below instead of bibmodel.
% 2. Run pdfLaTeX on the tex file.
% 3. Run BibTeX on the aux file.
% 4. Run LaTeX again a couple of times to get the correct 
%    cross-references.
%
% In many LaTeX editors, the above loop can be done just by pressing 
% a couple of buttons. 
\vspace{1cm}
  
  \bibliographystyle{plain}
  \bibliography{biblio}

% Two authors with different addresses: first option

% 
% \begin{address}
%   second author \\
%   first line of the address \\
%   second line \\
%   last line \\
%   \texttt{second author's email}
% \end{address}

% Two authors with different addresses: second option

% \begin{address}[2]
%   first author \\
%   first line of the address \\
%   second line \\
%   last line \\
%   \texttt{first author's email}
% \end{address}
% \begin{address}
%   second author \\
%   first line of the address \\
%   second line \\
%   last line \\
%   \texttt{second author's email}
% \end{address}

% Three authors. There are many cases. The simplest way to proceed: 
% put every author in a separate address environment. You can also,
% for example:
% 
% * join two consecutive addresses on the same line (see the second
%   option for two authors with different addresses)
%   
% * join several authors with the same address (see two authors
%   with the same address)

\end{document}